\documentclass[oneside,12pt]{amsart}
\usepackage[T1]{fontenc}
\usepackage{geometry}
\usepackage{microtype}
\usepackage{mathtools}
\usepackage{amssymb}
\usepackage{amsthm}
\usepackage{xcolor}
\usepackage{hyperref}
\usepackage{comment}
	\hypersetup{
		colorlinks=true,
		linkcolor=red,
		urlcolor=blue,
		citecolor=magenta
	}
\newcommand{\defining}[1]{\emph{#1}}			%%% markup for defining terms occurring in the text
\newcommand{\defn}{\coloneqq}			%%% markup for definition of an object (1)
\newcommand{\ndivides}{\nmid}			%%% a \ndivides b
\newcommand{\Iff}{if\textcompwordmark{f}}	%%% removing the ligature when typesetting iff

\newcommand{\Bigoh}{O}					%%% f = \Bigoh(g)		if f is bounded above by g (asymptotically)
\newcommand{\Littleoh}{o}				%%% f = \Littleoh(g)	if f is dominated by g (asymptotically)
\newcommand{\braces}[1]{\lbrace #1 \rbrace}		%%% curly brackets	{ }
\newcommand{\floor}[1]{\lfloor #1 \rfloor}		%%% floor function
\newcommand{\card}[1]{\lvert #1 \rvert}			%%% set cardinality

\newcommand{\st}{:}

\newcommand{\Set}[1]{\braces{ #1 }}%

\newcommand{\famF}{\mathcal{F}}	
\newcommand{\famG}{\mathcal{G}}
\newcommand{\famP}{\mathcal{P}}
\DeclareMathOperator{\core}{\mathsf{Cor}}
\DeclareMathOperator{\petal}{\mathsf{Pet}}
\DeclareMathOperator{\Tor}{\mathsf{Tor}}
\DeclareMathOperator{\sunset}{\mathsf{Set}}
\newcommand{\union}{\cup}
\newcommand{\Union}{\bigcup}
\newcommand{\disjunion}{\sqcup}
\newcommand{\disjUnion}{\bigsqcup}
\newcommand{\intersect}{\cap}
\newcommand{\Intersect}{\bigcap}
\newcommand{\Normal}{\text{nor}}
\newcommand{\Excep}{\text{exc}}

\theoremstyle{plain}
\newtheorem{theorem}{Theorem}[section]
\newtheorem{lemma}[theorem]{Lemma}
\newtheorem{corollary}[theorem]{Corollary}
\newtheorem{proposition}[theorem]{Proposition}
\newtheorem{question}[theorem]{Question}
\newtheorem{problem}[theorem]{Problem}
\newtheorem{claim}{Claim}

\theoremstyle{definition}
\newtheorem{definition}[theorem]{Definition}

\theoremstyle{remark}
\newtheorem{remark}[theorem]{Remark}
\newtheorem{observation}[theorem]{Observation}
\newtheorem{example}[theorem]{Example}

\begin{document}

\title{On hierarchically closed fractional intersecting families}

\author[N. Balachandran et al.]{Niranjan Balachandran}
\address{Department of Mathematics, Indian Institute of Technology Bombay, Mumbai, India}
\email{niranj@math.iitb.ac.in}

\author[]{Srimanta Bhattacharya}
\address{Department of Computer Science and Engineering, Indian Institute of Technology Palakkad, Palakkad, India}
\email{mail.srimanta@gmail.com}

\author[]{Krishn Vishwas Kher}
\address{Department of Engineering Science cum Computer Science and Engineering, Indian Institute of Technology Hyderabad, Hyderabad, India}
\email{es19btech11015@iith.ac.in}

\author[]{Rogers Mathew}
\address{Department of Computer Science and Engineering, Indian Institute of Technology Hyderabad, Hyderabad, India}
\email{rogers@cse.iith.ac.in}
\thanks{Research of Rogers Mathew is supported by a grant from the Science and Engineering Research Board, Department of Science and Technology, Govt.\ of India (project number:  MTR/2019/000550).}

\author[]{Brahadeesh Sankarnarayanan}
\address{Department of Mathematics, Indian Institute of Technology Bombay, Mumbai, India}
\email{bs@math.iitb.ac.in}
\thanks{Research of Brahadeesh Sankarnarayanan is supported by the National Board for Higher Mathematics (NBHM), Department of Atomic Energy (DAE), Govt.\ of India., and by the Industrial Research and Consultancy Centre (IRCC), Indian Institute of Technology Bombay, Mumbai, India.}

\date{November 8, 2023}
\subjclass{Primary 05D05; Secondary 05B99, 03E05}
\keywords{Fractional intersecting family, sunflower, intersection theorem}

\begin{abstract}
	For a set $L$ of positive proper fractions and a positive integer $r \geq 2$, a fractional $r$-closed $L$-intersecting family is a collection  $\mathcal{F} \subset \mathcal{P}([n])$ with the property that for any $2 \leq t \leq r$ and $A_1, \dotsc, A_t \in \mathcal{F}$ there exists $\theta \in L$ such that $\lvert A_1 \cap \dotsb \cap A_t \rvert \in \{ \theta \lvert A_1 \rvert, \dotsc, \theta \lvert A_t \rvert\}$.
	In this paper we show that for $r \geq 3$ and $L = \{\theta\}$ any fractional $r$-closed $\theta$-intersecting family has size at most linear in $n$, and this is best possible up to a constant factor.
	We also show that in the case $\theta = 1/2$ we have a tight upper bound of $\lfloor \frac{3n}{2} \rfloor - 2$ and that a maximal $r$-closed $(1/2)$-intersecting family is determined uniquely up to isomorphism.
\end{abstract}

\maketitle

\section{Introduction}

The theory of set systems with restricted intersection sizes is a classical and well-studied problem and the basic template of the problem is as follows.
Given a set $L$ of non-negative integers, determine the maximum size of a family $\famF\subset\famP([n])$ of subsets of $[n] \defn \Set{1,\dotsc,n}$ such that for distinct $A,B\in\famF$ we have $\card{A\intersect B}\in L$.
This problem has its origins in the de Bruijn--Erd\H{o}s theorem with further extensions including the Ray-Chaudhuri--Wilson inequality, the Frankl--Wilson inequality, and the Alon--Babai--Suzuki inequality among a host of other interesting results \cite{deBruijnErdos1948,Ray-ChaudhuriWilson1975,FranklWilson1981,FranklGraham1985,AlonBabaiEtAl1991,Ramanan1997,Snevily2003} and has spawned several variants, each with its own set of highlights and difficulties  besides ushering in a wide range of combinatorial and algebraic tools that are now an integral component of combinatorial techniques for extremal problems.

A recent variant~\cite{BalachandranMathewEtAl2019} of this problem, which is the principal focus of this paper, introduces the notion of \emph{fractional intersecting families} which goes as follows.
Suppose $L=\Set{\theta_1,\dotsc,\theta_\ell}$ is a set of proper positive fractions with $0<\theta_i=\frac{a_i}{b_i}<1$ and $\gcd(a_i,b_i)=1$ for each $i$.
We say that $\famF\subset\famP([n])$ is a \defining{fractional $L$-intersecting family} (or that $\famF$ is \defining{fractionally $L$-intersecting}) if for any two distinct sets $A,B\in\famF$ there exists $\theta\in L$ such that $\card{A\intersect B}\in\Set{\theta\card{A},\theta\card{B}}$.
The most natural question again is:
how large can a fractional $L$-intersecting family be?
This problem still remains unresolved;
the best known bounds are a poly-logarithmic factor away from optimal bounds~\cite{BalachandranMathewEtAl2019}.
The notion of fractional intersecting families has produced other related variants, including the notion of fractional $L$-intersecting families of vector spaces~\cite{MathewMishraEtAl2022} and fractional cross-intersecting families~\cite{MathewRayEtAl2021,WangHou2022}.
Attempts to obtain a linear upper bound for $\card{L}=1$ have led to conjectures on ranks of certain ensembles of matrices~\cite{BalachandranBhattacharyaEtAl2021,BalachandranBhattacharyaEtAl2022}, so the problem of fractional intersecting families has generated a considerable amount of interest.

In this paper we propose a more hierarchical extension of this notion of fractional intersecting families.
But before we get to the notion in more precise terms, we return to the original problem concerning the size of fractional intersecting families for some motivation.
For the rest of the paper, we shall always have $L=\Set{\theta}$ where $\theta=\frac{a}{b}$ is a proper positive fraction with $\gcd(a,b)=1$, and we shall also use the term ``\(\theta\)-intersecting'' interchangeably with ``\(L\)-intersecting''.

One of the main results in~\cite{BalachandranMathewEtAl2019} states that if $\famF$ is a fractional $L$-intersecting family with $L=\Set{\frac{a}{b}}$ then $\card{\famF}\leq \Bigoh_b(n\log n)$.
On the lower bound side, there are constructions of fractional $L$-intersecting families of size $\Omega(n)$.
For $\theta=\frac{1}{2}$, one can improve upon the constant a little more; there exist bisection closed families\footnote{When $\theta=1/2$ a fractional $L$-intersecting family is also called a bisection closed family in~\cite{BalachandranMathewEtAl2019}.} of size $\floor{\frac{3n}{2}}-2$.
What makes the problem of determining the size of maximal bisection closed families more interesting and intriguing is that there are non-isomorphic families of size $\floor{\frac{3n}{2}}-2$.
The simplest example (and an instructive one at that) is the following.

\begin{example}\label{Eg:sunflower}
	For the sake of simplicity, denote the set $\Set{x_1, \dotsc, x_{\ell}}$ by $x_1 \dotsm x_{\ell}$.
	Then, the family
	\[
		\famF =
			\begin{cases}
				\Set{12,13,\dotsc,1n,1234,1256,\dotsc,12(n-1)n}, & n \equiv 0 \pmod{2};\\
				\Set{12,13,\dotsc,1n,1234,1256,\dotsc,12(n-2)(n-1)}, & n \equiv 1 \pmod{2},
			\end{cases}
	\]
	is not only bisection closed, but also \defining{hierarchically bisection closed} in the following sense:
	for any sets $A_1, \dotsc, A_r\in\famF$ we also have $\card{A_1\intersect\dotsb\intersect A_r}\in\Set{\frac12 \card{A_1},\dotsc, \frac12 \card{A_r}}$.
	The easiest way to see this is to note that for this family, the subfamilies of sizes $2$ and $4$ are sunflowers, and also that any collection of subsets in $\famF$ have non-empty intersection.
\end{example}

The other known bisection closed families of size $\floor{\frac{3n}{2}}-2$ arise from a construction using Hadamard matrices, and do not  satisfy this stronger property.
\begin{example}\label{Eg:hadamard}
	Let \(H\) be an \(m \times m\) Hadamard matrix, i.e.\ a matrix whose entries lie in \(\Set{\pm 1}\), and with all the rows being mutually orthogonal.
	Assume that \(H\) is normalized so that the first row is the all-ones vector.
	Let \(J\) denote the \(m \times m\) all-ones matrix.
	Consider the matrix
	\[
		\begin{bmatrix}
			H & \phantom{-}H\\
			H & -H\\
			H & -J
		\end{bmatrix},
	\]
	and delete the first and \((2m+1)\)th rows.
	Viewing the remaining rows as the \(\pm1\) incidence vectors of subsets of \([2m]\), one can verify that this defines a family \(\famF \subset \famP([n])\) that is \(2\)-bisection closed, where \(n = 2m\).
	Since there are \(3m-2\) sets in \(\famF\), we have \(\card{\famF} = \frac{3n}{2} - 2\).
\end{example}
One of the principal reasons why a linear bound, let alone a tight bound, for the size of a bisection closed family is elusive is this diffusive nature of the known families of maximal size.
But since this last example seems to be structurally different from the others, it raises the following more natural question:
how large could a hierarchically bisection closed family be?
 
In order to make this precise, we make a formal definition. 
\begin{definition}\label{D:intersecting}
	Let \(r \geq 2\) and \(L = \Set{ \theta_{1}, \dotsc, \theta_{m} }\) be a set of fractions in \((0,1)\).
	So, \(\theta_{i} = a_{i}/b_{i}\) for some positive integers \(a_{i}, b_{i}\) such that \(\gcd(a_{i},b_{i}) = 1\), for each \(1 \leq i \leq m\).
	A family \(\famF\) of subsets of \([n]\) is called \defining{hierarchically \(r\)-closed \(L\)-intersecting} (or simply {\it $r$-closed $L$-intersecting}) if, for each \(2 \leq t \leq r\) and any \(t\) distinct sets \(A_{1},\dotsc,A_{t}\) in \(\famF\) we have \(\card{ \Intersect_{i=1}^{t} A_{i} } \in \Set{ \theta_{j}\card{A_{i}} \st 1 \leq i \leq t,\ 1 \leq j \leq m }\).
	
	When \(L = \Set{\theta}\), an \(r\)-closed \(L\)-intersecting family is also called an \defining{\(r\)-closed \(\theta\)-intersecting} family.
	In particular, when \(\theta = 1/2\), we call such a family \defining{\(r\)-bisection closed}.
\end{definition}
Note that	if a \(\theta\)-intersecting family is \(r\)-closed, then it is also \(s\)-closed for all \(2 \leq s \leq r\), which explains why we refer to such a family as hierarchically closed.

%As mentioned earlier, we shall always have  $|L|=1$ in the rest of the paper. \\

The natural question that arises is the following.
Suppose $r\geq 3$.
If $\famF\subset\famP([n])$ is $r$-closed $\theta$-intersecting, then determine the optimal upper bound for $\card{\famF}$.
Note that if $r=2$, then we are back to the case of fractional $L$-intersecting families, so it behooves us to set $r\geq 3$ if we hope to see any different emergent phenomenon arising from the definition.
And the main thesis of this paper is that setting $r\geq 3$ makes a big difference.

It is imperative to compare this notion with  another generalization that appears in~\cite{Mishra2019} which goes as follows.
For an integer $r\geq 2$, and $L$ as above, a family $\famF$ is said to be \defining{$r$-wise fractionally $L$-intersecting} if for any distinct $A_1,\dotsc,A_r\in\famF$ there exists $\theta\in L$ such that $\card{A_1\intersect\dotsb\intersect A_r}\in\Set{\theta\card{A_1},\dotsc,\theta\card{A_r}}$.
Again, the problem of determining the size of a maximum $r$-wise fractional $L$-intersecting family is optimally determined in~\cite{Mishra2019}  up to poly-logarithmic factors, and it appears that to get beyond the poly-logarithmic factor needs newer ideas (see~\cite{BalachandranMathewEtAl2019} for more details on this).
Our notion of $r$-closed $\theta$-intersecting is somewhat related and yet vastly different as the main results of our paper will attest. 
%it is akin to comparing the problem of maximal $t$-avoiding codes versus maximal $t+1$-intersecting codes. It might well be the case that asymptotically, these notions coincide, but that is yet to be determined.\\

We are now in a position to state the main results of the paper.
\begin{theorem}\label{T:normal}
	Let \(\famF\) be an \(r\)-bisection closed family over \([n]\), with \(r \geq 3\).
	Then,
	\begin{equation}\label{Eq:bound}
		\card{\famF} \leq \floor{\tfrac{3n}{2}} - 2\tag{\(*\)}
	\end{equation}
	for all \(n \geq 2\).
	Moreover:
	\begin{enumerate}
		\item (Tightness) For each \(n \geq 2\), there exists an \(r\)-bisection closed family \(\famF_{\max}\) over \([n]\) which attains the bound in \eqref{Eq:bound}.
		\item\label{T:normal2} (Uniqueness) For any family \(\famF\) over \([n]\) that attains the bound in \eqref{Eq:bound}, there is a permutation \(\sigma\) of \([n]\) such that \(\famF_{\max} = \sigma(\famF) \defn \Set{\sigma(A) \st A \in \famF}\), where \(\sigma(E) \defn \Set{ \sigma(a) \st a \in E}\) for any set \(E \in \famP([n])\).
		\item (Stability) There exists an absolute constant $C>0$ such that the following holds.
		If \(\card{\famF} \geq (\frac32 - \epsilon)n\) for some \(0 < \epsilon < 0.1\), then for some permutation \(\sigma\) of \([n]\),
		\[
			\card{\sigma(\famF) \setminus \famF_{\max}} < C\epsilon n.
		\]
	\end{enumerate}
\end{theorem}

When \(\famF\) is a general \(r\)-closed \(\theta\)-intersecting family, where \(\theta\) is not necessarily equal to \(1/2\), we do not have a tight upper bound on \(\card{\famF}\).
But, we are able to establish a linear upper bound on \(\card{\famF}\) even in this case.
\begin{theorem}\label{T:imin2b}
	Let \(\famF\) be an \(r\)-closed \(\theta\)-intersecting family over \([n]\), with \(r \geq 3\).
	Let \(\theta = a/b \in (0,1)\) with \(\gcd(a, b) = 1\), \(a, b > 0\).
	\begin{enumerate}
		\item If \(a > 1\), then \(\card{\famF} \leq 2\bigl(\frac{\ln(b)-\ln(a) + 1}{b-a}\bigr) (n-a) + 1\).
		
		\item If \(a = 1\), then we have two cases:
		\begin{enumerate}
			\item if \(b = 2\) and \(\famF\) contains a set of size \(2\), then \(\card{\famF} \leq (1 + \ln(2))(n-1) + 1\);
		
			\item otherwise, \(\card{\famF} \leq \bigl(\frac{2 \ln(b)}{b - 1}\bigr)(n-1) + 1\).
		\end{enumerate}
	\end{enumerate}
\end{theorem}

The rest of the paper is organized as follows.
We start with some preliminary results along with some terminology and develop some tools and lemmas in the next section. In Section~\ref{section:main_proofs}, we prove Theorem~\ref{T:imin2b}, and then use this to prove Theorem~\ref{T:normal}. 
We finally conclude with some remarks and open questions in Section~\ref{section:conclusion}.

%%%%%%%%%%%%%%%%%%%%%%%%%%%%%%%%%%%%%%%%%%%%%%%%%%%
\section{Preliminaries}

In what follows, we always assume that \(\famF\) is an \(r\)-closed \(\theta\)-intersecting family with \(r \geq 3\).
We denote by \(\famF(i)\) the collection of all \(i\)-element sets in \(\famF\), that is, \(\famF(i) \defn \famF \intersect \binom{[n]}{i}\).

Our first observation is that the possible sizes that could appear in any intersection of \(t\) sets (\(2 \leq t \leq r\)) in \(\famF\) is quite limited.

\begin{proposition}\label{P:small-intersection}
	Let \(2 \leq t \leq r\) and suppose  \(A_1,\dotsc,A_t \in \famF\) are distinct sets with \(\card{A_1} \leq \dotsb \leq \card{A_t}\).
	Then, \(\card{A_1 \intersect \dotsb \intersect A_t} \in \Set{\theta \card{A_1}, \theta \card{A_2}}\).
\end{proposition}
\begin{proof}
Since \(2 \leq t \leq r\), we have \(\theta \card{A_1} \leq \card{A_1 \intersect \dotsb \intersect A_t} \leq \card{A_1 \intersect A_2} \leq \theta \card{A_2}\), and so \(\card{A_1 \intersect \dotsb \intersect A_t} \in  \Set{\theta \card{A_1}, \theta \card{A_2}}\).
\end{proof}

Next, we show that one can often define a core of a set \(A \in \famF\) with certain nice properties.
\begin{definition}\label{D:tor}
	For \(A \in \famF\), define the set \(\Tor(A)\) of \(\theta\)-intersectors of \(A\) by
	\[
		\Tor(A) \defn \Set{ B \in \famF \st \card{B} \geq \card{A},\ \card{A \intersect B} = \theta \card{A}}.
	\]
\end{definition}
Note the condition \(\card{B} \geq \card{A}\) in the definition of \(\Tor(A)\).

\begin{proposition}\label{P:core-well-defined}
	If \(\Tor(A) \neq \emptyset\), then \(A \intersect B = A \intersect B'\) for all \(B, B' \in \Tor(A)\).
\end{proposition}
\begin{proof}
	We have \(\theta\card{A} \leq \card{A \intersect B \intersect B'} \leq \card{A \intersect B} = \card{A \intersect B'} = \theta \card{A}\).
	Hence, \(A \intersect B \intersect B' = A \intersect B\) and \(A \intersect B \intersect B' = A \intersect B'\).
	Thus, \(A \intersect B = A \intersect B'\). 
\end{proof}

\begin{definition}\label{D:core}
	For \(A \in \famF\) such that \(\Tor(A) \neq \emptyset\), define the \defining{core} of \(A\) by
	\[
		\core(A) \defn A \intersect B
	\]
	for any \(B \in \Tor(A)\).
\end{definition}

Proposition~\ref{P:core-well-defined} shows that Definition~\ref{D:core} is well-defined.
For a set \(A \in \famF\), \(\core(A)\) is not defined \Iff\ \(\Tor(A) = \emptyset\).
The next two results describe when this may happen.

\begin{proposition}\label{P:core-well-defined2.1}
	Let \(\card{\famF(i)} \geq 2\).
	Then \(\Tor(A) \neq \emptyset\) for all \(A \in \famF(i)\).
\end{proposition}
\begin{proof}
	If \(A,B \in \famF(i)\) are two distinct sets, then \(\card{A \intersect B} = \theta \card{A}\), so \(B \in \Tor(A)\).
	Hence, \(\Tor(A) \neq \emptyset\).
\end{proof}

\begin{corollary}\label{C:core-well-defined2.1.1}
	If \(A \in \famF(i)\) such that \(\Tor(A) = \emptyset\), then \(\famF(i) = \Set{A}\).
\end{corollary}

In fact, Proposition~\ref{P:core-well-defined2.1} implies that the family \(\famF\) is a union of uniform sunflowers.

\begin{definition}\label{D:sunflower}
	A family \(\famF\) of subsets of \([n]\) is called a \defining{sunflower} if, for \(C \defn \Intersect_{A \in \famF} A\), we have \(A \intersect B = C\) for all distinct \(A, B \in \famF\).
\end{definition}

\begin{lemma}\label{L:sunflower}
	Every nonempty \(\famF(i)\) is a sunflower.
\end{lemma}
%\begin{proof}
%	This is trivial if \(\card{\famF(i)} \leq 2\), so assume that \(\card{\famF(i)} = k \geq 3\).
%	We prove that \(\famF(i)\) is a sunflower by induction.
%	So, let \(2 \leq t \leq k - 1\) be an integer such that there is a \(t\)-subfamily \(\Set{A_{1}, \dotsc, A_{t}}\) of \(\famF(i)\) that is a sunflower.
%	Define \(\core_{t}(\famF(i)) \defn \Intersect_{j = 1}^{t} A_{j}\).
%	Note that \(\card{\core_{t}(\famF(i))} = \theta i\), and that \(\core_{t}(\famF(i)) = A_{j} \intersect A_{j'}\) for any \(1 \leq j < j' \leq t\).
%
%	Now, let \(A_{t+1} \in \famF(i)\) be any set distinct from \(A_{1}, \dotsc, A_{t}\).
%	Define \(\core_{t+1}(\famF(i)) \defn \Intersect_{j = 1}^{t+1} A_{j}\).
%	On the one hand, \(A_{j} \intersect A_{j'} \intersect A_{t+1} \subseteq A_{j} \intersect A_{j'} = \core_{t}(\famF(i))\) for any \(1 \leq j < j' \leq t\).
%	On the other hand, \(\card{A_{j} \intersect A_{j'} \intersect A_{t+1}} = \theta i = \card{\core_{t}(\famF(i))}\).
%	So, \(\core_{t}(\famF(i)) \subseteq A_{t+1}\).
%	
%	Next, let \(z \in (A_{j} \intersect A_{t+1}) \setminus \core_{t}(\famF(i))\) for some \(1 \leq j \leq t\).
%	Then, \(\card{A_{j} \intersect A_{t+1}} > \theta i\), which is a contradiction.
%	Hence, \(A_{1}, \dotsc, A_{t+1}\) is also a sunflower, with \(\core_{t}(\famF(i)) = \core_{t+1}(\famF(i))\).
%	By induction, \(\famF(i)\) is a sunflower.
%\end{proof}
\begin{proof}
	If \(\card{\famF(i)} \leq 2\), then this is trivial.
	Let \(\card{\famF(i)} \geq 3\).
	To show that \(\card{\famF(i)}\) is a sunflower, it suffices to show that \(\core(A) = \core(B)\) for any two sets \(A, B \in \famF(i)\). 
	The proof of Proposition~\ref{P:core-well-defined2.1} shows that \(A \in \Tor(B)\) and \(B \in \Tor(A)\) for any two sets \(A, B \in \famF(i)\).
	Hence, \(\core(A) = A \intersect B = B \intersect A = \core(B)\).
\end{proof}

\begin{remark}
	\hfill
	\begin{enumerate}
		\item Note that the set \(C\) in Definition~\ref{D:sunflower} is usually called the core of the sunflower.
		In particular, if the sunflower is a singleton set \(\Set{A}\), then \(C = A\).
	
		However, our definition of core is Definition~\ref{D:core}.
		This matches with the above notion when \(\card{\famF(i)} \geq 2\).
		But, when \(\famF(i) = \Set{A}\), \(\core(A)\) is either undefined (if \(\Tor(A) = \emptyset\)), or a subset of \(A\) having cardinality \(\theta i\) (if \(\Tor(A) \neq \emptyset\)).

		\item The property of being \(3\)-closed is crucially used in the proof of Proposition~\ref{P:core-well-defined}.
		Thus, if \(\famF\) is not \(3\)-closed, then Definition~\ref{D:core} cannot be made, and Lemma~\ref{L:sunflower} need not hold.
		Indeed, Example~\ref{Eg:hadamard} shows that there are \(2\)-bisection closed families that do not satisfy this property.
	\end{enumerate}
\end{remark}

We now establish some notations that we will use throughout the rest of this paper.
Let
\begin{align*}
	S &\defn \Set{ i \in [n] \st \famF(i) \neq \emptyset }, &i_{\min} &\defn \min(S),\\
	S_{\Normal} &\defn \Set{ i \in S \st \Tor(A) \neq \emptyset \text{ for all } A \in \famF(i) }, &i_{\max} &\defn \max(S_\Normal),\\
	S_{\Excep} &\defn \Set{ i \in S \st \Tor(A) = \emptyset \text{ for some } A \in \famF(i)}.
\end{align*}
Note that \(S = S_{\Normal} \disjunion S_{\Excep}\).
We say that \(\famF(i)\) is a \defining{normal} sunflower if \(i \in S_{\Normal}\), and we say that it is an \defining{exceptional} sunflower if \(i \in S_{\Excep}\).
Define \(\famF_{\Normal} \defn \Union_{i \in S_{\Normal}} \famF(i)\) and \(\famF_{\Excep} \defn \Union_{i \in S_{\Excep}} \famF(i)\).
Then, \(\famF = \famF_{\Normal} \disjunion \famF_{\Excep}\).
Define \(\petal(A) \defn A \setminus \core(A)\) for each \(A \in \famF_{\Normal}\).
For the sake of brevity, we also define the following:
\begin{align*}
	\sunset(\famF(i)) &\defn \Union_{A \in \famF(i)} A &\text{ for any } & i \in S,\\
	\petal(\famF(i)) &\defn \Union_{A \in \famF(i)} \petal(A)& \text{ for any } & i \in S_{\Normal}, \\
	\core(\famF(i)) &\defn \core(A) &\text{ for any } & A \in \famF(i),\ i \in S_{\Normal}.
\end{align*}
Furthermore, let
\[
	\famF(\geq i) \defn \Union_{j \geq i} \famF(j) \qquad\text{and}\qquad \famF(I) \defn \Union_{i \in I} \famF(i) \quad \text{for any } I \subset [n].
\]
Thus, we may also speak of \(\petal(\famF(\geq i))\) and \(\sunset(\famF(\geq i))\), as well as \(\petal(\famF(I))\) and \(\sunset(\famF(I))\) for any \(I \subset [n]\).

\begin{observation}\label{O:Sne}
	Proposition~\ref{P:core-well-defined2.1} and Corollary~\ref{C:core-well-defined2.1.1} show that if \(\Tor(A) \neq \emptyset\) for some \(A \in \famF(i)\), then \(i \in S_{\Normal}\), and if \(i \in S_{\Excep}\), then \(\card{\famF(i)} = 1\).
\end{observation}

\subsection{The structure of \texorpdfstring{\(\famF_{\Normal}\)}{F nor}}

The next few results describe the structure of the normal sunflowers in \(\famF\) in relation to the cores.

\begin{observation}\label{O:sunflower}
	The proof of Lemma~\ref{L:sunflower} shows that if \(A, B \in \famF_{\Normal}\) with \(\card{A} = \card{B}\), then \(\core(A) = \core(B)\).
\end{observation}

\begin{lemma}\label{L:core-inclusion1}
	If \(A, B \in \famF_{\Normal}\) with \(\card{A} < \card{B}\), then \(\core(A) \subsetneq \core(B)\).
\end{lemma}
\begin{proof}
	Let \(A' \in \Tor(A)\), \(B' \in \Tor(B)\).
	Consider \(A \intersect A' \intersect B = \core(A) \intersect B \subseteq \core(A)\).
	Since \(\theta \card{A} \leq \card{A \intersect A' \intersect B} \leq \card{\core(A)} = \theta \card{A}\), we have \(A \intersect A' \intersect B = \core(A)\) and \(\core(A) \subseteq B\). Since \(\card{B} \leq \card{B'}\), we can run the above argument with \(B'\) in place of \(B\) to show that \(\core(A) \subseteq B'\).
	Hence, \(\core(A) \subseteq B \intersect B' = \core(B)\).
	Lastly, \(\core(A) \neq \core(B)\) because \(\card{\core(A)} = \theta \card{A} \neq \theta \card{B} = \card{\core(B)}\).
\end{proof}

\begin{lemma}\label{L:petal-core}
	Suppose that \(i, j \in S\) such that \(i < \theta j\).
	If \(A \in \famF(i)\) and \(B \in \famF(j)\), then \(B \in \Tor(A)\).
	In particular, \(i \in S_{\Normal}\).
\end{lemma}
\begin{proof}
	Since \(\card{A \intersect B} \leq \card{A} < \theta j\), we must have \(\card{A \intersect B} = \theta i\).
	Hence, \(B \in \Tor(A)\).
	Thus, \(i \in S_{\Normal}\) by Observation~\ref{O:Sne}.
\end{proof}

\begin{lemma}\label{L:petal-core2}
	Let \(A \in \famF_{\Normal}\).
	If there exists \(B \in \famF(i_{\max})\) such that \(\petal(A) \intersect \core(B) \neq \emptyset\), then \(\core(B) \subseteq A\).
	Moreover, there is at most one set \(A \in \famF_{\Normal}\) for which this happens.
\end{lemma}
\begin{proof}
	Note that \(\card{A} < i_{\max}\) by Observation~\ref{O:sunflower}.
	Let \(C \in \Tor(B)\), and consider \(A \intersect B \intersect C = A \intersect \core(B) \subseteq \core(B)\).
	By Lemma~\ref{L:core-inclusion1}, \(\core(A) \subseteq \core(B)\), and \(\petal(A) \intersect \core(B) \neq \emptyset\) by assumption.
	Hence, \(\theta \card{A} < \card{A \intersect \core(B)}\), which implies that \(\theta i_{\max} \leq \card{A \intersect B \intersect C} = \card{A \intersect \core(B)} \leq \card{\core(B)} = \theta i_{\max}\).
	Thus, \(\core(B) \subseteq A\).
	
	Now, suppose that there exists \(A' \in \famF_{\Normal}\) distinct from \(A\) for which there exists \(B' \in \famF(i_{\max})\) such that \(\petal(A') \intersect \core(B') \neq \emptyset\).
	By Lemma~\ref{L:sunflower}, \(\core(B) = \core(B')\).
	So, \(\core(B) \subseteq A \intersect A'\), which implies that \(\card{A \intersect A'} \geq \theta i_{\max}\), a contradiction.
\end{proof}

Denote by \(E_{\Normal}\) the unique set \(A \in \famF_{\Normal}\) for which there exists \(B \in \famF(i_{\max})\) such that \(\petal(A) \intersect \core(B) \neq \emptyset\), whenever it exists.
Define \(\famF_{\Normal}^{*} \defn \famF_{\Normal} \setminus \Set{A \in \famF_{\Normal} \st A = E_{\Normal}}\).

\begin{corollary}\label{C:petal-core2}
	For all \(A,B \in \famF_{\Normal}^{*}\), \(\petal(A) \intersect \core(B) = \emptyset\).
\end{corollary}
\begin{proof}
	If \(\card{A} > \card{B}\), then \(\core(A) \supsetneq \core(B)\) by Lemma~\ref{L:core-inclusion1}, so \(\petal(A) \intersect \core(B) = \emptyset\).
	If \(\card{A} = \card{B}\), then this follows from Observation~\ref{O:sunflower}.
	Let \(\card{A} < \card{B}\), and suppose \(z \in \petal(A) \intersect \core(B)\).
	Then, by Lemma~\ref{L:core-inclusion1}, \(z \in \core(B')\) for any \(B' \in \famF(i_{\max})\).
	Hence, \(\petal(A) \intersect \core(B') \neq \emptyset\), which implies by Lemma~\ref{L:petal-core2} that \(A = E_{\Normal}\), a contradiction.
\end{proof}

Lemma~\ref{L:core-inclusion1} and Corollary~\ref{C:petal-core2} say that \(\famF_{\Normal}^*\) has the following structure:
the cores of \(\famF_{\Normal}^*\) form an increasing chain, and any petal is disjoint from every core.
In fact, these two results can be used to show that, for \(\famF_{\Normal}^*\), ``\(r\)-closed'' is equivalent to ``\(s\)-closed'' for any \(r, s \geq 3\).

\begin{proposition}\label{P:allbisect}
	\(\famF_{\Normal}^{*}\) is \(s\)-closed \(\theta\)-intersecting for all \(s \geq 2\).
\end{proposition}
\begin{proof}
	It suffices to show this for all \(s > r \geq 3\), and by induction it is enough to show this for \(s = r+1\).
	Let \(A_{1},\dotsc,A_{r+1} \in \famF_{\Normal}^{*}\) be any \(r+1\) distinct sets.
	Without loss of generality, suppose that \(\card{A_{1}} \leq \dotsb \leq \card{A_{r+1}}\).
	
	First, suppose that \(\card{A_{i}} = \card{A_{j}}\) for some \(i < j\).
	Then, \(\core(A_{1}) \subseteq \Intersect_{k=1}^{r+1} A_{k} \subseteq \core(A_{i})\) by Lemma~\ref{L:core-inclusion1} and Observation~\ref{O:sunflower}.
	But, by Corollary~\ref{C:petal-core2}, \(\petal(A_{1}) \intersect \core(A_{i}) = \emptyset\).
	Hence, \(\Intersect_{k=1}^{r+1} A_{k} = \core(A_{1})\).
	Thus, \(\card{\Intersect_{k=1}^{r+1} A_{k}} = \theta \card{A_{1}}\).
	So, we are done in this case.
	
	Next, suppose that \(\card{A_{i}} < \card{A_{j}}\) for all \(i < j\).
	Consider \(U = A_{1} \intersect \dotsb \intersect A_{r}\) and \(V = A_{1} \intersect \dotsb \intersect A_{r-1} \intersect A_{r+1}\).
	By Proposition~\ref{P:small-intersection}, we know that \(\card{U}, \card{V} \in \Set{\theta \card{A_{1}}, \theta \card{A_{2}}}\).
	Also, \(\card{U \intersect V} \leq \min\Set{\card{U},\card{V}}\).
	Note that \(U \intersect V = A_{1} \intersect \dotsb \intersect A_{r+1}\).
	
	By Lemma~\ref{L:core-inclusion1}, \(\core(A_{1}) \subseteq U \intersect V\).
	So, if \(\card{U} = \theta \card{A_{1}}\) or \(\card{V} = \theta \card{A_{1}}\), then \(\theta \card{A_{1}} \leq \card{U \intersect V} \leq \theta \card{A_{1}}\), and we are done in this case.
	So, assume that \(\card{U} = \theta \card{A_{2}} = \card{V}\).
	Consider \(U \subseteq A_1 \intersect A_2\).
	Since \(\theta \card{A_2} = \card{U} \leq \card{A_1 \intersect A_2} \leq \theta \card{A_2}\), we have \(U = A_1 \intersect A_2\).
	Similarly, \(V \subseteq A_1 \intersect A_2\) and \(\theta \card{A_2} = \card{V} \leq \card{A_1 \intersect A_2} \leq \theta \card{A_2}\), so \(V = A_1 \intersect A_2\).
	Hence, \(U \intersect V = A_1 \intersect A_2\), and \(\card{U \intersect V} = \card{A_1 \intersect A_2} = \theta \card{A_2}\), so we are done.
\end{proof}

The final result of this section provides a linear upper bound on the size of \(\famF\) when \(\famF = \famF_{\Normal}^*\) and \(\Tor(A) = \Set{ B \in \famF \st \card{B} \geq \card{A}}\) for every \(A \in \famF\).
Also, the proof technique will be used later on in the proof of Theorem~\ref{T:imin2b} in Section~\ref{S:proofs}.

\begin{lemma}\label{L:tor-full}
	Suppose that for all \(A, B \in \famF_{\Normal}^*\) such that \(\card{A} < \card{B}\), we have \(B \in \Tor(A)\).
	Then, \(\card{\famF_{\Normal}^*} \leq \floor{\frac{n-a}{b-a}}\).
\end{lemma}
\begin{proof}
	For simplicity of notation, assume that \(\famF = \famF_{\Normal}^*\).
	Suppose that \(S = \Set{i_1, \dotsc, i_k}\) with \(i_1 < \dotsb < i_k\).
	Let \(C \defn \core(\famF(i_k))\).
	Define \(Y_j \defn \sunset(\famF(i_j)) \setminus C\) for each \(1 \leq j \leq k\).
	By Lemma~\ref{L:core-inclusion1} and Corollary~\ref{C:petal-core2}, \(Y_j = \petal(\famF(i_j))\) for each \(1 \leq j \leq k\).
	Since \(B \in \Tor(A)\) whenever \(A \in \famF(i_j)\) and \(B \in \famF(i_{j'})\) for \(j < j'\), we must have \(\petal(A) \intersect \petal(B) = \emptyset\).
	Thus, \(Y_j \intersect Y_{j'} = \emptyset\) for all \(j \neq j'\).
	Now, notice that
	\[
		\card{\famF(i_j)} = \frac{\card{Y_j}}{(1-\theta)i_j},
	\]
	since the petals in \(\famF(i_j)\) are pairwise disjoint sets with each having size \((1-\theta)i_j\).
	Thus,
	\[
		\card{\famF} = \sum_{j = 1}^k \card{\famF(i_j)} = \sum_{j=1}^k \frac{\card{Y_j}}{(1-\theta)i_j}.
	\]
	We also have \(\sum_{j=1}^k \card{Y_j} \leq n - \card{C} = n - \theta i_k\).
	It is now easy to see that \(\card{\famF}\) is maximized when \(\card{Y_j} = (1-\theta)i_j\) for \(2 \leq j \leq k\), and \(\card{Y_1}\) is the largest integer \(\leq n - \theta i_k - \sum_{j=2}^k (1-\theta)i_j\) which is divisible by \((1-\theta)i_1\).
	Thus, the maximum of \(\card{\famF}\) taken as \(S\) varies over all subsets of \([n]\) of size \(k\), with \(k\) varying from \(1\) to \(n\), occurs when \(k=1\) and \(i_1 = b\), where \(\theta = a/b\) in least form, \(a, b > 0\).
	This maximum is easily seen to be \(\floor{\frac{n-a}{b-a}}\).
\end{proof}

\subsection{The structure of \texorpdfstring{\(\famF_{\Excep}\)}{F exc}}

The next few results describe the structure of the exceptional sunflowers in \(\famF\) in relation to the cores.

\begin{lemma}\label{L:core-well-defined2.2}
	Suppose that \(S_{\Normal} \neq \emptyset\).
	Let \(i \in S_{\Excep}\) such that \(i > i_{\max}\).
	If \(\famF(i) = \Set{A}\), then \(\core(\famF(i_{\max})) \subseteq A\).
\end{lemma}
\begin{proof}
	Let \(B \in \famF(i_{\max})\) and \(C \in \Tor(B)\).
	Consider \(A \intersect B \intersect C = A \intersect \core(B) \subseteq \core(B)\).
	Since, \(\theta \card{B} \leq \card{A \intersect B \intersect C} \leq \card{\core(B)} = \theta \card{B}\), we have \(\core(B) \subseteq A\), as required.
\end{proof}

\begin{lemma}\label{L:core-well-defined2.3}
	Suppose that \(S_{\Normal} \neq \emptyset\).
	Let \(i \in S_{\Excep}\) such that \(i < i_{\max}\).
	If \(\famF(i) = \Set{A}\), then, either \(\card{A \intersect \core(\famF(i_{\max}))} = \theta i\), or \(\core(\famF(i_{\max})) \subseteq A\).
	Moreover, there is at most one \(i < i_{\max}\) such that the latter case holds.
\end{lemma}
\begin{proof}
	Let \(B \in \famF(i_{\max})\) and \(C \in \Tor(B)\).
	Consider \(A \intersect B \intersect C = A \intersect \core(B) \subseteq \core(B)\).
	If \(\card{A \intersect B \intersect C} < \theta i_{\max}\), then we must have \(\card{A \intersect \core(B)} = \theta i\), which is the former case.	
	If \(\card{A \intersect B \intersect C} = \theta i_{\max}\), then \(A \intersect B \intersect C = \core(B)\), since \(\card{\core(B)} = \theta i_{\max}\).
	Hence, \(\core(B) \subseteq A\), which is the latter case.
	Lastly, suppose for the sake of contradiction that there exists \(i' \in S_{\Excep}\), \(i' \neq i\), such that \(i' < i_{\max}\), \(\famF(i') = \Set{A'}\), and \(\core(B) \subseteq A'\).
	Then, \(A \intersect A' \supseteq \core(B)\), so \(\card{A \intersect A'} \geq \theta i_{\max}\), which is a contradiction.
\end{proof}

Denote by \(E_{\Excep}\) the unique set in \(\famF_{\Excep}\) such that \(\card{E_{\Excep}} < i_{\max}\) and \(\core(\famF(i_{\max})) \subseteq E_{\Excep}\), whenever it exists.

\begin{lemma}\label{L:theta}
	Let \(\theta = a/b\), \(\gcd(a,b) = 1\).
	Let \(A \in \famF\) such that \(b \ndivides \card{A}\).
	Then, \(A \in \famF_{\Excep}\), and there is at most one such set \(A\) in \(\famF\).
	Moreover, if \(S_{\Normal} \neq \emptyset\), then \(\core(\famF(i_{\max})) \subseteq A\).
\end{lemma}
\begin{proof}
	For any \(A_{1} \in \famF\) distinct from \(A\), we must have \(\card{A \intersect A_{1}} = \theta \card{A_{1}}\), since \(\theta \card{A}\) is not an integer.
	So, \(\Tor(A) = \emptyset\), which implies that \(A \in \famF_{\Excep}\).
	If there were another such set \(A'\), then \(\card{A \intersect A'}\) can be neither \(\theta \card{A}\) nor \(\theta \card{A'}\), which is a contradiction.
	
	Let \(S_{\Normal} \neq \emptyset\), \(B \in \famF(i_{\max})\), and \(C \in \Tor(B)\).
	Consider \(A \intersect B \intersect C = A \intersect \core(B)\).
	Since, \(\card{A \intersect B \intersect C} \neq \theta \card{A}\), we have \(\theta i_{\max} \leq \card{A \intersect B \intersect C} \leq \card{\core(B)} = \theta i_{\max}\).
	Hence, \(A \intersect B \intersect C = \core(B)\), which implies that \(\core(B) \subseteq A\).
	
\end{proof}

Denote by \(E_{\theta}\) the unique set in \(\famF\) such that \(b \ndivides \card{E_{\theta}}\) (where \(\theta = a/b\), \(\gcd(a,b) = 1\)), whenever it exists.
Define \(\famF_{\Excep}^{*} \defn \famF_{\Excep} \setminus \Set{A \in \famF_{\Excep} \st A = E_{\Excep} \text{ or } E_{\theta}}\).
Define \(\famF^* \defn \famF_{\Normal}^{*} \union \famF_{\Excep}^{*}\).

\subsection{The structure of \texorpdfstring{\(\famF^{*}\)}{F*}}

\begin{observation}\label{O:divide}
	If \(\theta = a/b\), \(\gcd(a,b) = 1\), then \(\card{A} \equiv 0 \pmod{b}\) for all \(A \in \famF^*\).
\end{observation}

\begin{proposition}\label{P:exceptions}
	\(\card{\famF^*} \leq \card{\famF} \leq \card{\famF^*} + 1\).
\end{proposition}
\begin{proof}
	It suffices to show that at most one of \(E_{\Normal}\), \(E_{\Excep}\), and \(E_{\theta}\) can belong to the family \(\famF\).
	If \(S_{\Normal} = \emptyset\), then neither \(E_{\Normal}\) nor \(E_{\Excep}\) can exist by definition.
	So, suppose that \(S_{\Normal} \neq \emptyset\).
	Then, \(\core(\famF(i_{\max})) \subseteq E_{\Normal}\), \(E_{\Excep}\), and \(E_{\theta}\) by Lemmas~\ref{L:petal-core2}, \ref{L:core-well-defined2.3}, and~\ref{L:theta}, respectively.
	Hence, the size of the intersection of any two of these sets must be at least \(\card{\core(\famF(i_{\max}))} = \theta i_{\max}\), which is neither \(\theta \card{E_{\Normal}}\), nor \(\theta \card{E_{\Excep}}\), nor \(\theta \card{E_{\theta}}\), which is a contradiction.
\end{proof}

%%%%%%%%%%%%%%%%%%%%%%%%%%%%%%%%%%%%%%%%%%%%%%%%%%
\section{Proofs of Theorems \ref{T:normal} and \ref{T:imin2b}}\label{S:proofs}
\label{section:main_proofs}
Assume that \(\famF = \famF^*\).
Lemma~\ref{L:petal-core} motivates us to partition the family \(\famF\) as \(\famF = \disjUnion_{k \geq 0} \famF(I_k)\), where \(I_k \defn (i_{\min}/\theta^{k-1}, i_{\min}/\theta^k]\) for \(k \geq 1\), and \(I_0 \defn \Set{i_{\min}}\).
Suppose that \(S_{\Normal} \neq \emptyset\).
Let \(C \defn \core(\famF(i_{\max}))\).
Define \(Y_k \defn \sunset(\famF(I_k)) \setminus C\).
\begin{observation}\label{O:disjoint}
	If \(v > u + 1\), then \(Y_u \intersect Y_v = \emptyset\).
\end{observation}
\begin{proof}
	It suffices to show that if \(A \in \famF(i)\) (\(i \in I_{u} \intersect S\)) and \(B \in \famF(j)\) (\(j \in I_{v} \intersect S\)), then \(A \intersect B \subseteq C\).
	It follows from the definitions of \(I_{k}\), \(k \geq 0\), that \(i < \theta j\) for any such \(i\) and \(j\).
	Thus, by Lemma~\ref{L:petal-core}, \(B \in \Tor(A)\), so \(A \intersect B = \core(A)\).
	Hence, by Lemma~\ref{L:core-inclusion1}, \(A \intersect B \subseteq C\).
\end{proof}

\begin{observation}\label{O:Ybound}
	\begin{align*}
		\sum_{k \text{ odd}} \card{Y_k} &\leq n - \card{C} \leq n - \theta i_{\min},\\
		\sum_{k \text{ even}} \card{Y_k} &\leq n - \card{C} \leq n - \theta i_{\min}.
	\end{align*}
\end{observation}
\begin{proof}
	This is immediate from the previous observation.
\end{proof}

\begin{observation}\label{O:Fbound}
	Let \(i \in I_{k}\).
	Then,
	\[
		\card{\famF(i)} \leq \frac{\card{Y_{k}}}{(1 - \theta)i}.
	\]
%	We may (over)count the number of petals in a sunflower \(\famF(i)\), \(i \in I_k\), by partitioning \(Y_k\) into pairwise disjoint sets, each of size \((1-\theta)i\).
\end{observation}
\begin{proof}
%	Since \(0 \leq \card{Y_{k}}\), it suffices to assume that \(i \in S\).
%	
	Let \(i \in S_{\Normal}\).
	By Lemma~\ref{L:core-inclusion1} and Corollary~\ref{C:petal-core2}, \(A \setminus C = \petal(A)\) for all \(A \in \famF(i)\), so \(Y_{k} \supseteq \petal(\famF(i))\).
	By Lemma~\ref{L:sunflower}, \(\petal(A) \intersect \petal(A') = \emptyset\) for all distinct \(A, A' \in \famF(i)\).
	Hence, \(\card{Y_{k}} \geq \card{\petal(\famF(i))} = \sum_{A \in \famF(i)} \card{\petal(A)}\).
	Since \(\card{\petal(A)} = (1-\theta)i\) for all \(A \in \famF(i)\), we are done.
	
	Let \(i \in S_{\Excep}\) and \(\famF(i) = \Set{A}\).
	First, consider the case when \(i > i_{\max}\).
	Since \(Y_{k} \supseteq A \setminus C\), and \(C \subseteq A\) by Lemma~\ref{L:core-well-defined2.2}, we have \(\card{Y_{k}} \geq \card{A} - \card{C} = i - \theta i_{\max} > i - \theta i\).
	So, we are done.
	Next, consider the case when \(i < i_{\max}\).
	Since we assume that \(\famF = \famF^{*}\), we have \(\card{A \intersect C} = \theta i\) by Lemma~\ref{L:core-well-defined2.3}.
	Hence, \(\card{A \setminus C} = i - \theta i\).
	Since \(Y_{k} \supseteq A \setminus C\), we are done.
\end{proof}

We also need the following result.
\begin{lemma}\label{L:log}
	Let \(\eta > 1\), and let \(m \geq 1\) be an integer.
	Consider the sequence \((s_k)_{k \geq 1}\) given by
	\[
		s_k \defn \frac{1}{\floor{m\eta^{k-1}}+1} + \frac{1}{\floor{m\eta^{k-1}}+2} + \dotsb + \frac{1}{\floor{m\eta^k}}.
	\]
	Then, \(\lim_{k \to \infty} s_k = \ln(\eta)\).
	
	When \(\eta\) is an integer, the sequence \((s_k)_{k \geq 1}\) is monotonically increasing to \(\ln(\eta)\). 
	In general, \(s_k < \ln(\eta)+ \frac{1}{m}\) for all \(k \geq 1\).
\end{lemma}
\begin{proof}
	Let \(H_n\) denote the \(n\)th harmonic number, \(H_n = \sum_{i = 1}^{n} 1/i\).
	It is well-known that \(\lim_{n \to \infty} (H_n - \ln(n)) = \gamma\), the Euler--Mascheroni constant.
	Hence,
	\[
		s_k = H_{\floor{m\eta^k}} - H_{\floor{m\eta^{k-1}}} = \ln\left(\frac{\floor{m\eta^k}}{\floor{m\eta^{k-1}}}\right) + \epsilon(\floor{m\eta^k}) - \epsilon(\floor{m\eta^{k-1}}),
	\]
	where \(\lim_{n \to \infty} \epsilon(n) = 0\).
	Since
	\[
		\eta - \frac{1}{m\eta^{k-1}} < \frac{\floor{m\eta^k}}{\floor{m\eta^{k-1}}} < \frac{\eta}{1-\frac{1}{m\eta^{k-1}}},
	\]
	we have \(\lim_{k \to \infty} s_k = \ln(\eta)\).
	
	When \(\eta\) is an integer, the monotonicity of the sequence \((s_k)_{k \geq 1}\) is a corollary of the following more general observation, where \(n \geq 1\) is any integer:
	\[
		\sum_{i = n + 1}^{\eta n} \frac{1}{i} < \sum_{i = n + 1}^{\eta n} \frac{1}{i} + \left(\frac{1}{\eta n + 1} - \frac{1}{\eta n + \eta}\right) + \dotsb + \left(\frac{1}{\eta n + \eta - 1} - \frac{1}{\eta n + \eta}\right) = \sum_{i = (n+1) + 1}^{\eta(n+1)} \frac{1}{i}.
	\]

	To show that \(s_k < \ln(\eta) + \frac{1}{m}\) for all \(k \geq 1\), observe that
	\begin{align*}
		s_k &< \int_{\floor{m\eta^{k-1}}}^{\floor{m\eta^{k}}} \frac{1}{t} dt\\
			%&= \ln(\floor{m\eta^{k}}) - \ln(\floor{m\eta^{k-1}})\\
			&\leq \ln(m\eta^k) - \ln(\floor{m\eta^{k-1}})\\
			&= \ln(\eta) + \ln\left(\frac{m\eta^{k-1}}{\floor{m\eta^{k-1}}}\right)\\
			%&< \ln(\eta) + \ln\left(1 + \frac{1}{\floor{m\eta^{k-1}}}\right)\\
			&< \ln(\eta) + \frac{1}{\floor{m\eta^{k-1}}}\\
			&\leq \ln(\eta) + \frac{1}{m}.
\end{align*}
\end{proof}

\subsection{Proof of Theorem~\ref{T:imin2b}}
Now, we are ready to prove Theorem \ref{T:imin2b}.
\begin{proof} 
	We assume throughout that \(\famF = \famF^{*}\), since it suffices to compute \(\card{\famF^*}\) by Proposition~\ref{P:exceptions}.

 	First, observe that if \(\famF_{\Normal} = \emptyset\), then only \(\famF(I_0)\) and \(\famF(I_1)\) can be nonempty by Lemma~\ref{L:petal-core}.
	Furthermore, each nonempty \(\famF(i)\) is a singleton set.
	Therefore, \(\card{\famF} \leq \frac{1}{b}\left( \frac{i_{\min}}{\theta} - i_{\min}\right) + 1\), which is maximized when \(n = \frac{i_{\min}}{\theta}\).
	Hence, this gives the bound \(\card{\famF} \leq \floor{\frac{(1-\theta)n}{b}} + 1\), which is stronger than those in the statement of Theorem~\ref{T:imin2b}.
	
	For the rest of the proof, suppose that \(\famF_{\Normal} \neq \emptyset\).
%	First, let \(\theta > 1/2\).
%	Let \(A, B \in \famF\) such that \(\card{A} < \card{B}\).
%	{\color{red}{Then, \(B \in \Tor(A)\), because \(\theta\card{B} > \frac12\card{B} \geq \card{A}\), where the latter inequality holds by Observation~\ref{O:divide}.}}
%	In particular, \(\card{\famF_{\Excep}} \leq 1\).
%	So, by Lemma~\ref{L:tor-full}, \(\card{\famF} \leq \floor{\frac{n-a}{b-a}} + 1\).
%
%	Next, let \(\theta \leq 1/2\).
	Let \(i_{\min} = mb\) for some \(m \geq 1\) by Observation~\ref{O:divide}.
	For \(k \geq 1\), we have
	\[
		\card{\famF(I_k)} = \sum_{i \in I_{k} \intersect S} \card{\famF(i)} \leq \sum_{i \in I_k \intersect S} \frac{\card{Y_k}}{(1-\theta) i} \leq
			\begin{dcases}
				\frac{\card{Y_k}}{b-a} \left( \ln(\theta^{-1}) + \frac{1}{m} \right), & a > 1;\\
				\frac{\card{Y_k}}{b-1} (\ln(b)), & a = 1,
			\end{dcases}
	\]
	from Observations~\ref{O:divide} and \ref{O:Fbound}, as well as Lemma~\ref{L:log}.
	For \(k = 0\), we have
	\[
		\card{\famF(I_0)} = \card{\famF(i_{\min})} \leq \frac{\card{Y_0}}{(1-\theta)mb} \leq
			\begin{dcases}
				\frac{\card{Y_0}}{b-a} \left( \ln(\theta^{-1}) + \frac{1}{m} \right), & a > 1;\\
				\frac{\card{Y_0}}{b-1} \left( \frac{1}{m} \right), & a = 1,
			\end{dcases} 
	\]
	from Observation~\ref{O:Fbound} and Lemma~\ref{L:log}.
	Since
	\[
		\card{\famF} = \sum_{k \geq 0} \card{\famF(I_k)} = \sum_{k \text{ odd}} \card{\famF(I_k)} + \sum_{k \text{ even}} \card{\famF(I_k)},
	\]
	we get the bound
	\begin{equation}\label{Eq:part2}
		\card{\famF} \leq 2\left(\frac{\ln(b)-\ln(a) + 1}{b-a}\right) (n-\card{C})
		%\leq 2\left(\frac{\ln(b)-\ln(a) + 1}{b-a}\right) (n-a)
	\end{equation}
	when \(a > 1\) by applying Observation~\ref{O:Ybound}.
	
	When \(a = 1\), we need to compare the term \(1/m\) appearing in the bound for \(\famF(I_0)\) with the term \(\ln(b)\) appearing in the bound for \(\famF(I_k)\) for \(k\) even: since \(1/m > \ln(b)\) if and only if \(m = 1\) and \(b = 2\), and this happens if and only if \(\theta = 1/2\) and \(i_{\min} = 2\), we get
	\[
		\sum_{k \text{ odd}} \card{\famF(I_k)} \leq \frac{\ln(b)}{b-1} \sum_{k \text{ odd}} \card{Y_k}, \qquad \sum_{k \text{ even}} \card{\famF(I_k)} \leq
		\begin{dcases}
			{\frac{1}{2-1}}\sum_{k \text{ even}} \card{Y_k}, & \theta = 1/2,\ i_{\min} = 2;\\[1em]
			\frac{\ln(b)}{b-1} \sum_{k \text{ even}} \card{Y_k}, & \text{otherwise}.
		\end{dcases}
	\]
	Thus, by Observation~\ref{O:Ybound},
	\begin{equation}\label{Eq:part3}
		\card{\famF} \leq
			\begin{dcases}
				(1 + \ln(2))(n-\card{C}), & \theta = 1/2 \text{ and } i_{\min} = 2;\\
				\left(\frac{2 \ln(b)}{b - 1}\right)(n-\card{C}), & \text{ otherwise}.
			\end{dcases}
	\end{equation}
	The result now follows immediately from \eqref{Eq:part2} and \eqref{Eq:part3}.
\end{proof}

%%%%%%%%%%%%%%%%%%%%%%%%%%%%%%%%%%%%%%%%%%%%%%%%%%
\subsection{Proof of Theorem~\ref{T:normal}}
%We now turn to the proof of Theorem \ref{T:normal}.
We begin with an outline of the proof of Theorem~\ref{T:normal} before presenting the details.
Since the theorem is easily verified for \(n = 2,3\), we may assume that \(n \geq 4\).
It also suffices to assume that \(\famF = \famF^*\) by Proposition~\ref{P:exceptions}.
First, we show that the upper bound on \(\card{\famF}\) holds when \(S = S_{\Normal} = \Set{2,4}\).
Second, we show that if \(S_{\Normal} \nsupseteq \Set{2,4}\), then \(\famF\) cannot be an extremal family.
Finally, we show that if \(S_{\Normal} \supsetneq \Set{2, 4}\), then we can get a family that is strictly larger than \(\famF\) by removing all the sets of sizes greater than \(4\) and adding new sets of sizes \(2\) and \(4\).
The uniqueness and stability are then easily verified, thus completing the proof.

\begin{proof}
	Example~\ref{Eg:sunflower} shows that there exists an \(r\)-bisection closed family \(\famF\) such that \(\card{\famF} = \floor{\frac{3n}{2}}-2\) for any \(n \geq 2\), so the bound \eqref{Eq:bound}, which we shall establish below, is in fact tight.
	For the rest of the proof, we assume that \(n \geq 4\) and that \(\famF = \famF^*\). 
	\begin{claim}\label{Cl:unique}
		If \(S = S_{\Normal} = \Set{2,4}\), then \(\card{\famF} \leq \floor{\frac{3n}{2}} - 3\).
	\end{claim}
	Let \(\core(\famF(2)) = \Set{a_1}\) and \(\core(\famF(4)) = \Set{a_1,a_2}\).
	By Corollary~\ref{C:petal-core2} it follows that \(\card{\famF(4)} \leq \floor{\frac{n-2}{2}}\) and \(\card{\famF(2)} \leq n-2\) so \(\card{\famF} = \card{\famF(2)} + \card{\famF(4)} \leq  \floor{\frac{3n}{2}} - 3\).
	\begin{claim}
		If \(S_{\Normal} \nsupseteq \Set{2,4}\), then \(\famF\) is not an extremal family.
	\end{claim}
	Suppose for the sake of contradiction that \(\famF\) is extremal.
	Let \(C \defn \core(\famF(i_{\max}))\). 

		If \(S = \Set{2, 4}\) but \(S_{\Excep} \neq \emptyset\), then clearly there cannot be more than \(n\) sets in the family \(\famF\), contradicting its extremality.
		So, assume that \(S \neq \Set{2,4}\).
	
		Theorem~\ref{T:imin2b} already shows that \(\card{\famF} < \floor{\frac{3n}{2}} - 3\) for a bisection closed family unless \(i_{\min} = 2\). 
		So, suppose that \(2 \in S\).
	
		If \(2 \in S_{\Excep}\), then there cannot be any \(i \in S\) such that \(i > 4\) by Lemma~\ref{L:petal-core}.
		So, \(S = \Set{2} = S_{\Excep}\), but this implies that \(\card{\famF} = 1\), contradicting the extremality of \(\famF\).
		Hence, \(2 \in S_{\Normal}\).
	
		Next, if \(4 \not\in S\), then by Lemma~\ref{L:petal-core}, \(A \intersect B = \core(A)\) for all \(A \in \famF(2)\), \(B \in \famF(\geq 6)\).
		If \(n = 4\), then \(\famF(\geq 6) = \emptyset\), so we must have \(S = \Set{2}\).
		However, this contradicts the extremality of \(\famF\), as we have seen earlier, so assume that \(n \geq 6\).
		Let \(m_1 = \card{\petal(\famF(2))}\) and \(m_2 = \card{\sunset(\famF(\geq 6))}\).
		Then, \(m_1 + m_2 \leq n\), and \(\card{\famF} \leq m_1 + \floor{2 \ln(2) (m_2-\card{C})} \leq 1 + \floor{2\ln(2) (n - 4)}\) by \eqref{Eq:part3}.
		This is less than \(\floor{\frac{3n}{2}}-3\), which contradicts the extremality of \(\famF\).
		So, \(4 \in S\).
		
		Lastly, if \(4 \in S_{\Excep}\), then \(S \subset \Set{2,4,6,8}\) by Lemma~\ref{L:petal-core}.
		Suppose that \(\famF(8) \neq \emptyset\).
		Then, if \(\famF(4) = \Set{A}\), we must have \(\card{A \intersect B} = \frac12\card{B} = \card{A}\) for any \(B \in \famF(8)\).
		Hence, \(A \subset B\) for all \(B \in \famF(8)\).
		So, if \(8 \in S_{\Normal}\), then \(A = \core(B)\), implying that \(A = E_{\Excep}\).
		This contradicts that \(\famF = \famF^*\), so \(8 \not\in S_{\Normal}\).
		But then \(\famF(4)\) and \(\famF(8)\) together contain at most two sets, and it is easy to see by a similar argument as in the previous case that \(\card{\famF}\) is strictly less than \(\floor{\frac{3n}{2}} - 3\), which contradicts the extremality of \(\famF\).
		So, \(4 \in S_{\Normal}\).
%}}
	
	To quickly summarize the above observations, \(S_{\Normal} \supseteq \Set{2, 4}\) for any extremal family \(\famF\).
	We will now show that if \(S_{\Normal} \supsetneq \Set{2,4}\), then \(\famF\) is not extremal. 
	Assume that \(\famF\) is an extremal family having the maximum number of sets of size \(2\).
	\begin{claim}
		If there exists \(a \in \petal(\famF(2)) \intersect B\) for some \(B \in \famF(\geq 4)\), then \(a \in \petal(\famF(4))\).
	\end{claim}
	This follows from Observation~\ref{O:disjoint} and Corollary~\ref{C:petal-core2}.
	\begin{claim}\label{Cl:35}
		\(\card{\petal(\famF(2)) \setminus \petal(\famF(4))} \leq 1\).
	\end{claim}
	Suppose for the sake of contradiction that \(a_1, a_2 \in \petal(\famF(2)) \setminus \petal(\famF(4))\) such that \(a_1 \neq a_2\).
	Define \(B' \defn \core(\famF(4)) \union \Set{a_1, a_2}\) and \(\famF' \defn \famF \union \Set{B'}\).
%	By Corollary~\ref{C:petal-core2}, \(B' \in \famF'(4)\).
	By Observation~\ref{O:disjoint}, \(a_1, a_2 \not\in \sunset(\famF'(\geq 6))\), so \(\famF'\) is \(r\)-bisection closed.
	But, \(\card{\famF'} > \card{\famF}\), which contradicts the maximality of \(\famF\).
	
	\begin{claim}\label{Cl:disjoint}
		For each \(B \in \famF(4)\), \(\petal(B) \intersect \petal(\famF(2)) = \emptyset\) or \(\petal(B)\). 
	\end{claim}
	Let \(a \in \petal(B) \intersect \petal(\famF(2))\), and let \(b \in \petal(B)\) such that \(b \neq a\).
	Suppose for the sake of contradiction that \(b \not\in \petal(\famF(2))\).
	If \(b \not\in \petal(A)\) for any \(A \in \famF\) distinct from \(B\), then we contradict the maximality of \(\famF\) as before by considering the family \(\famF' \defn \famF \union \Set{A'}\), where \(A' = \core(\famF(2)) \union \Set{b}\).
	So, \(b \in \petal(\famF(\geq 6))\) by Corollary~\ref{C:petal-core2}.
	Note that \(b \not\in \petal(\famF(\geq 10))\) by Observation~\ref{O:disjoint}.
	Also, if \(b \in \petal(\famF(A))\) for some \(A \in \famF(8)\), then we must have \(B \subset A\); in particular, \(a \in A\), which is not possible by Observation~\ref{O:disjoint}.
	Hence, \(b \in \petal(A)\) for some \(A \in \famF(6)\), which is also unique by Lemma~\ref{L:sunflower}. 
	Now, consider the family \(\famF'' \defn (\famF \setminus \Set{A}) \union \Set{A''}\), where \(A'' \defn \core(\famF(2)) \union \Set{b}\).
	Again, the property of being \(r\)-bisection closed is preserved, and \(\card{\famF''} = \card{\famF}\), but \(\card{\famF''(2)} > \card{\famF(2)}\), which is a contradiction.\\
	
	We now partition the family \(\famF\) into two disjoint nonempty subfamilies as follows:
	let \(\famG_1\) be the subfamily consisting of the sets in \(\famF(2)\) as well as those sets \(B\) in \(\famF(4)\) such that \(\petal(B) \intersect \petal(\famF(2)) \neq \emptyset\), and let \(\famG_2\) be the subfamily of \(\famF\) containing the remaining sets.
	Let \(m_1 = \card{\petal(\famG_1)}\) and \(m_2 = \card{\sunset(\famG_2)}\).
	By Claim~\ref{Cl:disjoint}, \(\petal(A) \intersect B = \emptyset\) for all \(A \in \famG_1\) and \(B \in \famG_2\).
	So, \(m_1 + m_2 \leq n\).
	Also, \(i_{\min}(\famG_2) \geq 4\), and \(\card{C} \geq 3\) since \(S \supsetneq \Set{2, 4}\).
	Observe that \(\card{\famG_1} = \floor{\frac{3m_1}{2}}\) and \(\card{\famG_2} \leq \floor{2 \ln(2) (m_2 - 3)}\) by \eqref{Eq:part3}.
	But then \(\card{\famF} = \card{\famG_{1}} + \card{\famG_{2}} \leq \floor{\frac{3m_1}{2}} + \floor{2 \ln(2) (m_2 - 3)} < \floor{\frac{3n}{2}} - 3\) since \(m_2 \geq 6\), which contradicts the extremality of \(\famF\).
	This completes the proof of the bound \eqref{Eq:bound}.
	The tightness, uniqueness, and stability are now easily verified:
	\begin{enumerate}
		\item As noted before, the family constructed in Example~\ref{Eg:sunflower} is tight for the upper bound \eqref{Eq:bound}.
		Call that family \(\famF_{\max}\).
		
		Note that \(\famF_{\max} = \famF_{\max}(2) \disjunion \famF_{\max}(4)\), and that \(\famF_{\max}\) is \(r\)-bisection closed for any \(r \geq 2\) because, for any family of subsets of \([n]\) consisting only of sets of sizes \(2\) and \(4\), ``\(r\)-bisection closed'' and ``intersecting'' are equivalent properties.
		Also note that \(E_\Normal = \Set{1,2}\) belongs to the family \(\famF_{\max}\).

		\item The proof of the upper bound \eqref{Eq:bound} shows that if \(\famF\) is an extremal \(r\)-bisection closed family, then \(S_\Normal = \Set{2, 4}\).
		Furthermore, Claim~\ref{Cl:unique} shows that for any extremal \(\famF\) we must have \(\card{\famF^*} = \floor{\frac{3n}{2}}-3\), and in particular \(\card{\famF^*(2)} = n-2\) and \(\card{\famF^*(4)} = \floor{\frac{n-2}{2}}\).
		That is, assuming \(\core(\famF^*(2)) = \Set{a_1}\) and \(\core(\famF^*(4)) = \Set{a_1,a_2}\), the sets in \(\famF^*(2)\) are precisely all those obtained by taking the union of \(\Set{a_1}\) with singleton sets \(\Set{b}\) such that \(b \neq a_1,a_2\), and the sets in \(\famF^*(4)\) are precisely all those obtained by taking the union of \(\Set{a_1,a_2}\) with two-element sets \(\Set{b_1,b_2}\) that are pairwise disjoint from each other as well as from \(\Set{a_1,a_2}\).
		Since \(\famF^*\) is an intersecting family, it is \(r\)-bisection closed, too.
		A moment's reflection shows that this family \(\famF^*\) can be obtained simply by applying an appropriate permutation of \([n]\) to \(\famF_{\max}^*\).
%		: the choices made in the construction --- picking the cores of the \(2\)-element and \(4\)-element sets, as well as picking \(\floor{\frac{n-2}{2}}\) pairwise disjoint two-element subsets from the remaining elements in \([n]\) --- completely determine the required permutation.
		
		To complete the analysis, observe that \(\famG \defn \famF^* \union \Set{\core(\famF(4))}\) is also \(r\)-bisection closed, and the permutation of \([n]\) that mapped \(\famF^*\) to \(\famF_{\max}^*\) also maps \(\famG\) to \(\famF_{\max}\).
		Clearly, \(E_{\Normal}(\famG) = \core(\famF(4))\).
		To show that \(\famG = \famF\), we verify that neither \(E_\Excep\) nor \(E_\theta\) can belong to \(\famF\).
		Suppose \(E_\Excep \in \famF\).
		Then \(\core(\famF(4)) \subsetneq E_\Excep\).
		But, if \(\Set{a} \neq \core(\famF(2))\), then \(a \in \petal(A)\) for some \(A \in \famF(2)\).
		In particular, we must have \(A \intersect E_\Excep = A\) which forces \(E_{\Excep} \in \famF(4)\), but this is a contradiction.
		The same argument also shows that \(E_\theta \not\in \famF\), and this completes the proof of uniqueness of the extremal family.
		
		\item Theorem~\ref{T:imin2b} and the proof of the upper bound \eqref{Eq:bound} show that \(\card{\famF} < 2\ln(2)(n-1) + 1\) for any \(r\)-bisection closed family \(\famF\) that is not extremal.
		Since \(\frac{3}{2}-2\ln(2) \approx 0.11\), the claim follows.
	\end{enumerate}
\end{proof}

%%%%%%%%%%%%%%%%%%%%%%%%%%%%%%%%%%%%%%%%%%%%%%%%%
\section{Concluding remarks}
\label{section:conclusion}
We ignore all floors and ceilings here for simplicity.
%{\color{red}{
\begin{itemize}
	\item While Theorem \ref{T:normal} considers the maximum size among all possible $r$-bisection closed families, it is possible to consider a more constrained problem:
	
	\begin{problem}\label{Prob:imin}
		For an integer $k \geq 2$, determine the maximum size of an $r$-bisection closed family $\famF$ with $i_{\min}(\famF^*) \geq k$.
	\end{problem}
	
	Theorem \ref{T:imin2b} establishes a linear upper bound, and it is not hard to construct a heirarchically bisection closed family of size at least $(2n-k-4)\left(\frac{1}{k} + \frac{1}{k+2} + \frac{1}{k+4}\right)$ when $k \geq 4$.
	Our methods in this paper suggest that all the possible set sizes must lie in the range $[k,2k]$ for an optimal family.
	There could be more than three distinct set sizes in an optimal family, though it seems rather unlikely that sets of all possible sizes in this range can be attained.
	Settling this question fully may require other new ideas.
	
	\item While Theorem \ref{T:normal} gives a tight result for $\theta=1/2$, the bound in Theorem \ref{T:imin2b} in the general case is far from best possible.
	Again, one can mimic the construction for $\famF_{\max}$ to get $r$-closed $\theta$-intersecting  families of size $(n-2a)\left( \frac{1}{b-a} + \frac{1}{2(b-a)} \right)$ if $\theta=\frac{a}{b}$, but this is not best possible in general.
%	For instance, if $\theta=1/4$, one can get a hierarchically closed $\theta$-intersecting family $\famF$ of size $(n-4)\left( \frac{1}{3} + \frac{1}{6} + \frac{1}{12} \right)$ (ignoring floors and ceilings), and for which $i_{\min}(\famF^*) = 4$.
	If $\theta = 1/b$ for $b$ odd, then one can get a heirarchically closed $\theta$-intersecting family $\famF$ of size $(n-3)\left( \frac{1}{b-1} + \frac{1}{2(b-1)} + \frac{1}{3(b-1)} \right)$.
	If $\theta = 1/b$ for $b$ even, then in general one can get a heirarchically closed $\theta$-intersecting family $\famF$ of size $(n-4)\left( \frac{1}{b-1} + \frac{1}{2(b-1)} + \frac{1}{4(b-1)} \right)$.
	Similar constructions can be made in general when $a \neq 1$.
	The methods in this paper suggest that the best bound ought to be attained when $i_{\min}(\famF^*)$ is as small as possible, i.e.\ $i_{\min}(\famF^*) = b$ when $\theta = a/b$ in least form, but a complete answer seems beyond the scope of the methods in this paper.
	
	\begin{problem}\label{Prob:theta}
		For a fraction $\theta = a/b \in (0,1)$, determine the maximum size of an $r$-closed $\theta$-intersecting family $\famF$.
	\end{problem}
	
	\item The following general question naturally arises from the above two problems, and we make the explicit statement for the sake of completeness:
	
	\begin{problem}\label{Prob:imin-theta}
		For a fraction $\theta = a/b \in (0,1)$ and an integer $k \geq b$, determine the maximum size of an $r$-closed $\theta$-intersecting family $\famF$ with $i_{\min}(\famF^*) = k$.
	\end{problem}
	
	\item Another interesting question arises as an artifact of our proof ideas. If $\famF=\famF_{\Excep}$ then the proof of Theorem \ref{T:imin2b} also shows that \(\card{\famF} \leq \left(\frac{1-\theta}{b}\right)n + 2\).
	But, it appears that this bound is far from best possible, and we believe that in this case $\card{\famF}=\Bigoh(\sqrt{n})$.
	Since the notion of an exceptional family seems a bit contrived, a more natural question is the following:
	
	\begin{question}\label{Prob:excep}
		Suppose $\famF=\Set{A_1,\dotsc,A_m}$ is an $r$-closed $\theta$-intersecting family with $\card{A_i} < \card{A_j}$ whenever $i<j$.
		Is $\card{\famF} \leq \Bigoh(\sqrt{n})$?
	\end{question}

	One indication that this bound is the correct order comes from the situation when $\card{A_i\intersect A_j} = \theta \card{A_i}$ whenever $i<j$.
	This setup is similar to that in Lemma~\ref{L:tor-full}, but under the additional constraint that there is at most one set of any fixed size.
	Indeed, in this case, a straightforward inductive argument shows that  $\card{\Union_{i=1}^k A_i} \geq k^2$, and that gives the bound stated.
	But in the general case, the methods developed in this paper seem to fall short of being able to settle this conjecture in the affirmative.
	The following weaker version of the above question could prove to be more amenable to investigation:
	
	\begin{question}\label{Prob:excep2}
		Suppose $\famF=\Set{A_1,\dotsc,A_m}$ is an $r$-closed $\theta$-intersecting family with $\card{A_i} < \card{A_j}$ whenever $i<j$.
		Is $\card{\famF} \leq \Littleoh(n)$?
	\end{question}
\end{itemize}
%}}

%%%%%%%%%%%%%%%%%%%%%%%%%%%%%%%%%%%%%%%%%%%%%%%%%
% The following optional unnumbered section is where you put personal acknowledgements,
% research grant support, and similar things.  Do not put them on the front page.
\subsection*{Acknowledgements}

We thank the reviewers for their careful reading of the paper, and for correcting our statement and proof of Theorem~\ref{T:imin2b}, as well as suggesting improvements in the presentation.

%The research of Rogers Mathew is supported by a grant from the Science and Engineering Research Board, Department of Science and Technology, Govt.\ of India (project number:  MTR/2019/000550).
%The research of Brahadeesh Sankarnarayanan is supported by the National Board for Higher Mathematics (NBHM), Department of Atomic Energy (DAE), Govt.\ of India, and by the Industrial Research and Consultancy Centre (IRCC), Indian Institute of Technology Bombay, Mumbai, India.

%BIBLIOGRAPHY
% You do not have to use the same format for your references, but 
%    include everything in this file.
% If you use BibTeX to create a bibliography, copy the .bbl file into here.
% We recommend you use \doi{...} and \arxiv{...} like the examples below,
% as they give a short display form with an active link to the full url.

\appendix
\section{Addendum}

{\small Theorem~\ref{T:normal}\eqref{T:normal2} says that any hierarchically \(r\)-bisection closed family \(\mathcal{F}\) over \([n]\) (for \(r \geq 3\)) that attains equality in the bound
\begin{equation}\label{Eq:*}
	\card{\mathcal{F}} \leq \floor{3n/2} - 2\tag{\(*\)}
\end{equation}
is the family \(\mathcal{F}_{\max}\) of Example~\ref{Eg:sunflower}, up to permutations of \([n]\).
In the proof of Theorem~\ref{T:normal}\eqref{T:normal2}, we merely wrote that, ``The proof of the upper bound \eqref{Eq:bound} shows that if \(\mathcal{F}\) is an extremal \(r\)-bisection closed family, then \(S_{\text{nor}} = \Set{2, 4}\).''
%However, the proof of the upper bound \eqref{Eq:*} only establishes a weaker statement, namely that any extremal \(r\)-bisection closed family \(\mathcal{F}\) \emph{for which \(\card{\mathcal{F}(2)}\) is maximum} satisfies \(S_{\text{nor}} = \Set{2,4}\).
%To complete the proof of Theorem~4(2) it remains to rule out the existence of extremal families for which \(\card{\mathcal{F}(2)}\) is not maximum, which we do so in this addendum.
However, the details require some filling in, which we do so in this addendum.\\
}

To show that there is a unique extremal family \(\mathcal{F}\) (up to permutations of \([n]\)) that attains the bound \eqref{Eq:*}, we first show that, among the families satisfying \(\mathcal{F} = \mathcal{F}^{*}\), the extremal ones have size \(\floor{3n/2} - 3\).
So, assume that \(\mathcal{F} = \mathcal{F}^{*}\) is extremal over \([n]\).

Claims~\ref{Cl:unique}--\ref{Cl:35} hold for any such \(\mathcal{F}\).
We restate a couple of these claims here:
\setcounter{claim}{1}
\begin{claim}\label{c33}
	\(S_{\text{nor}} \supseteq \Set{2,4}\).
\end{claim}
\begin{claim}\label{c34}
	If there exists \(b \in \mathsf{Pet}(\mathcal{F}(2)) \intersect A\) for some \(A \in \mathcal{F}(\geq 4)\), then \(A \in \mathcal{F}(4)\) and \(b \in \mathsf{Pet}(A)\).
\end{claim}
\setcounter{claim}{4}
An additional hypothesis was introduced in:
\begin{claim}\label{c36}
	Let \(\mathcal{F}\) be an extremal family for which \(\card{\mathcal{F}(2)}\) is maximum.
	Then for each \(B \in \mathcal{F}(4)\), \(\card{\mathsf{Pet}(B) \intersect \mathsf{Pet}(\mathcal{F}(2))} \in \Set{0, 2}\).
\end{claim}
Using Claim~\ref{c36} we showed that if \(\mathcal{F}\) is any extremal family \emph{for which \(\card{\mathcal{F}(2)}\) is maximum}, then \(S_{\text{nor}} = \Set{2, 4}\).
This was used to establish that \(\card{\mathcal{F}^{*}} \leq \floor{3n/2} - 3\) for any \(r\)-bisection closed family \(\mathcal{F}\) over \([n]\), as well as the following (weaker) uniqueness result (cf.~Theorem~\ref{T:normal}\eqref{T:normal2}):
\begin{lemma}\label{l}
	Let \(\mathcal{F}\) be an extremal \(r\)-bisection closed family over \([n]\) for which \(S_{\text{nor}} = \Set{2,4}\).
	Then, there is a permutation \(\sigma\) of \([n]\) such that \(\sigma(\mathcal{F}) = \mathcal{F}_{\max}\).
	In particular, if \(\mathcal{F}\) is an extremal family for which \(\card{\mathcal{F}(2)}\) is maximum, then \(\sigma(\mathcal{F}) = \mathcal{F}_{\max}\) for some permutation \(\sigma\) of \([n]\).
\end{lemma}
Note that \(\card{\mathcal{F}(2)} \leq n-1\) for any \(\mathcal{F}\), and equality holds for the extremal family \(\mathcal{F}_{\max}\).
Now, we reformulate Claim~\ref{c36} to avoid any extra assumptions on the size of \(\mathcal{F}(2)\):
\begin{claim}\label{c}
%	Let \(\mathcal{F}\) be an extremal family over \([n]\).
	Let \(B \in \mathcal{F}(4)\) and \(\mathsf{Pet}(B) = \Set{a,b}\). Then:
	\begin{enumerate}
		\item\label{1} \(\card{\Set{a,b} \intersect \mathsf{Pet}(\mathcal{F}(2))} \in \Set{0,2}\), or
		\item\label{2} \(\card{\Set{a,b} \intersect \mathsf{Pet}(\mathcal{F}(2))} = 1\), and if \(b \in \mathsf{Pet}(\mathcal{F}(2))\), then there is a unique set \(A \in \mathcal{F}(\geq 6)\) such that \(a \in A\). Moreover, \(A \in \mathcal{F}(6)\).
	\end{enumerate}
\end{claim}
\begin{proof}
%	Let \(B \in \mathcal{F}(4)\), and let \(\mathsf{Pet}(B) = \Set{a,b}\).
	Suppose that \(b \in \mathsf{Pet}(\mathcal{F}(2))\) and \(a \notin \mathsf{Pet}(\mathcal{F}(2))\).
	If \(a \notin B'\) for any \(B' \in \mathcal{F}\) distinct from \(B\), then we contradict the extremality of \(\mathcal{F}\) as follows: the family \(\mathcal{F}' \defn \mathcal{F} \union \Set{A'}\), where \(A' \defn \mathsf{Cor}(\mathcal{F}(2)) \union \Set{a}\), is \(r\)-bisection closed and satisfies \(\card{\mathcal{F}'} > \card{\mathcal{F}}\).
	
	So, there is a set \(A \in \mathcal{F}\) distinct from \(B\) for which \(a \in A\).
	In particular, \(A \in \mathcal{F}(\geq 6)\).
	Note that \(\mathsf{Cor}(B) \union \Set{a} \subseteq A\), so \(\card{A \intersect B} \geq 3 > \frac12 \card{B}\).
	Thus, \(\card{A \intersect B} = \frac12\card{A}\).
	So, if \(A \in \mathcal{F}(\geq 8)\), then in fact \(A \in \mathcal{F}(8)\) and \(B \subseteq A\).
	But this implies that \(b \in A\), which contradicts Claim~\ref{c34}.
	Thus, \(A \in \mathcal{F}(6)\).
	
	Lastly, if \(\mathcal{F}(6)\) is a singleton, then \(A\) is clearly unique, and if there are at least two sets in \(\mathcal{F}(6)\), then \(a \notin A'\) for any \(A' \in \mathcal{F}(6)\) distinct from \(A\) because \(\mathcal{F}(6)\) is a sunflower and \(a \in \mathsf{Pet}(A)\).
\end{proof}
Now, in terms of Claim~\ref{c} we have (without any change in the proof):
\begin{corollary}\label{cor}
	If Claim~\ref{c}(\ref{1}) holds for all \(B \in \mathcal{F}(4)\), then \(S_{\text{nor}} = \Set{2,4}\).
\end{corollary}

We are now ready to prove:
\begin{proposition}\label{P}
	There is no extremal family \(\mathcal{F}\) over \([n]\) for which \(\card{\mathcal{F}(2)} < n - 1\).
\end{proposition}
\begin{proof}
	Suppose for the sake of contradiction that \(\mathcal{F}\) is an extremal family over \([n]\) for which \(\card{\mathcal{F}(2)} < n - 1\).
	Then, \(S_{\text{nor}} \supsetneq \Set{2, 4}\) by Lemma~\ref{l} and Claim~\ref{c33}.
	Also, Claim~\ref{c}(\ref{2}) holds for some \(B \in \mathcal{F}(4)\) by Corollary~\ref{cor} and Lemma~\ref{l}.
	
	Now, let \(\mathcal{F}_{0} \defn \mathcal{F}\).
	For \(n \in \mathbb{N}\), if the extremal \(r\)-bisection closed family \(\mathcal{F}_{n}\) has been defined, and there is a set \(B_{n} \in \mathcal{F}_{n}(4)\) for which Claim~\ref{c}(\ref{2}) holds, then we define \(\mathcal{F}_{n+1}\) as follows.
	Let \(\mathsf{Pet}(B_{n}) = \Set{a_{n},b_{n}}\) with \(b_{n} \in \mathsf{Pet}(\mathcal{F}_{n}(2))\).
	Let \(A_{n}\in\mathcal{F}_{n}(6)\) be the unique set in \(\mathcal{F}_{n}(\geq 6)\) such that \(a_{n} \in A_{n}\).
	Then, define \(\mathcal{F}_{n+1} \defn (\mathcal{F}_{n} \setminus \Set{A_{n}}) \union \Set{A_{n}'}\), where \(A_{n}' \defn \mathsf{Cor}(\mathcal{F}_{n}(2)) \union \Set{a_{n}}\).
	Note that \(\mathcal{F}_{n+1}\) is also an \(r\)-bisection closed family that is extremal, since \(\card{\mathcal{F}_{n}} = \card{\mathcal{F}_{n+1}}\).
	
	Applying this procedure inductively by starting with \(\mathcal{F}_{0} \defn \mathcal{F}\), for some \(N \in \mathbb{N}\) we get an extremal family \(\mathcal{F}' = \mathcal{F}_{N}\) such that Claim~\ref{c}(\ref{1}) holds for all \(B' \in \mathcal{F}'(4)\).
	Hence, by Corollary~\ref{cor}, \(\mathcal{F}'\) has only two normal sunflowers, namely \(\mathcal{F}'(2)\) and \(\mathcal{F}'(4)\).
	Since the only sets from \(\mathcal{F}\) that were thrown out in the construction of \(\mathcal{F}'\) were those of size \(6\), \(\mathcal{F}\) has only three normal sunflowers, namely \(\mathcal{F}(2)\), \(\mathcal{F}(4)\), and \(\mathcal{F}(6)\).
	Now, let \(B \in \mathcal{F}(6)\), and let \(\mathsf{Pet}(B) = \Set{a,b,c}\).
	Define \(\mathcal{G} = (\mathcal{F}^{*} \setminus \Set{B}) \union \Set{D_{a},D_{b},D_{c}}\), where \(D_{i} \defn \mathsf{Cor}(\mathcal{F}(2)) \union \Set{i}\), for \(i \in \Set{a,b,c}\).
	Then, \(\mathcal{G}\) is an \(r\)-bisection closed family for which \(\card{\mathcal{G}} \geq \card{\mathcal{F}} + 1\), contradicting the extremality of \(\mathcal{F}\).
\end{proof}

This completes the proof of Theorem~\ref{T:normal}\eqref{T:normal2} that the family \(\mathcal{F}_{\max}\) over \([n]\) of Example~\ref{Eg:sunflower} is the unique extremal \(r\)-bisection closed family (up to permutations of \([n]\)).

\end{document}